\documentclass[12pt,compress]{article}
\usepackage{amsmath, amsthm}
\usepackage{amssymb}

\usepackage{mathrsfs}
\usepackage{scrextend}
\usepackage{enumerate}
\usepackage{longtable,booktabs,makecell,multirow,tabularx}
\usepackage{graphicx}

\usepackage{color}
\usepackage[colorlinks,bookmarksopen,bookmarksnumbered,citecolor=black,linkcolor=black,urlcolor=black]{hyperref}
\usepackage{hyperref}
\usepackage{geometry}
\newtheorem{theorem}{Theorem}[section]
\newtheorem{lemma}[theorem]{Lemma}
\newtheorem{corollary}[theorem]{Corollary}

\newtheorem{proposition}[theorem]{Proposition}

\newtheorem{theoremA}{Theorem}

\theoremstyle{definition}

\newtheorem{example}[theorem]{Example}

\numberwithin{equation}{section}

\def\UU{{\mathcal U}}

\def\MM{{\mathcal M}}

\usepackage[numbers,sort&compress]{natbib}

\usepackage{indentfirst}
\begin{document}
\begin{sloppypar}	

	\title{\textbf{A note on the $ \Pi $-property of some subgroups of finite groups}}
	
	\author{Zhengtian Qiu,  Guiyun Chen  and Jianjun Liu\footnote{Corresponding author.}\vspace{0.5em}\\
		{\small School of Mathematics and Statistics, Southwest University,}\\
		{\small  Chongqing 400715, People's Republic of China
		}\vspace{0.5em}\\
		{\small  E-mail addresses: qztqzt506@163.com,  gychen1963@163.com, liujj198123@163.com}}
	
	\date{}
	\maketitle
	
	\begin{abstract}
	Let $ H $ be a subgroup of a finite group $ G $. We say that $ H $ satisfies the $ \Pi  $-property in $ G $ if for any chief factor $ L / K $ of $ G $, $ |G/K : N_{G/K}(HK/K\cap L/K )| $ is a $ \pi (HK/K\cap L/K) $-number. In this paper,  we obtain some criteria for the $ p $-supersolubility or $ p $-nilpotency of a finite group and extend some known results by concerning some subgroups that satisfy the $ \Pi $-property.
	\end{abstract}

	\renewcommand{\thefootnote}{\fnsymbol{footnote}}
	
	
	
	\section{Introduction}
	
    All groups considered in this paper are finite.
    We use conventional notions as in \cite{Huppert-1967}. Throughout the paper,  $ G $ always denotes a finite group, $ p $ denotes a fixed prime, $ \pi $ denotes some set of primes  and $ \pi(G) $ denotes the set of all primes dividing $ |G| $. An integer $ n $ is called a $ \pi $-number if all prime divisors of $ n $ belong to $ \pi $.

    Suppose that $ P $ is a $ p $-group for some prime $ p $.  Let $ \MM(P) $ be the set of all
    maximal subgroups of $ P $. Let $ d $ be the smallest generator number of $ P $, i.e., $ p^{d} = |P/\Phi(P)| $, where $ \Phi(P) $ is the Frattini  subgroup of $ P $. According to \cite{Li-He-2008}, $ \MM _{d}(P) = \{P_{1},...,P_{d}\} $ is  a subset  of $ \MM(P) $ such that  $ \mathop{\cap}\limits_{i=1}^{d} P_{i}=\Phi(P) $. Notice that the subset $ \MM_{d}(P) $ is not unique for a fixed $p$-group  $ P $ in general. Li and He \cite{Li-He-2008} showed that $ |\MM(P)|\gg |\MM_{d}(P)| $,  because $$ \lim\limits_{d \to \infty}\frac{|\MM(P)|}{|\MM_{d}(P)|}=\lim\limits_{d \to \infty}\frac{\frac{p^{d}-1}{p-1}}{d}=\infty. $$

    In \cite{li-2011}, Li introduced  the concept of the $\Pi$-property  of subgroups of finite groups. Let $ H $ be a subgroup of a group $ G $. We say  that $ H $ satisfies the $ \Pi  $-property in $ G $ if for any   chief factor $ L/K $ of $ G $, $ |G/K : N_{G/K}(HK/K\cap L/K )| $ is a $ \pi (HK/K\cap L/K) $-number. There are many examples of embedding properties of
    subgroups implying the possession of the $ \Pi $-property (see Section \ref{S4}).

    In this note, we obtain some criteria for the $ p $-supersolubility or $ p $-nilpotency of a group $ G $  by assuming that some $ p $-subgroups of $ N $ satisfy the $ \Pi $-property in $ G $, where $ N $ is a normal subgroup of $ G $ such that $ G/N $ is $ p $-supersoluble. More precisely, we prove the following results:

    \begin{theoremA}\label{first}
    Let $ G $ be a $ p $-soluble group, and let $ P $ be a Sylow $ p $-subgroup
    of $ G $, where $ p $ is a prime dividing the order of $ G $. Assume that $ N $ is a normal subgroup of $ G $ such that $ G/N $ is $ p $-supersoluble. Then $ G $ is $ p $-supersoluble if and only if  $ P_{i}\cap N $ satisfies the $ \Pi  $-property in $ G $ for every member $ P_{i} $ of some fixed $ \MM_{d}(P) $.
    \end{theoremA}


  \begin{theoremA}\label{second}
  Let $ P $ be a Sylow $ p $-subgroup of a group $ G $ for some prime $ p\in \pi(G) $. Assume that $ N $ is a normal subgroup of $ G $ such that $ G/N $ is $ p $-supersoluble. Then $ G $ is $ p $-nilpotent if and only if $ N_{G}(P) $ is $ p $-nilpotent and  $ P_{i}\cap N $ satisfies the $ \Pi $-property in $ G $ for every member $ P_{i} $ of some fixed $ \MM_{d}(P) $.
  \end{theoremA}

  \begin{theoremA}\label{third}
 Let $ p $ be a prime dividing the order of a group $ G $  with $ (|G|, p-1) = 1 $ and $ P \in {\rm Syl}_{p}(G) $. Assume that $ N $ is a normal subgroup of $ G $ such that $ G/N $ is $ p $-nilpotent. Then $ G $ is $ p $-nilpotent if and only if $ P_{i}\cap N $ satisfies the $ \Pi $-property in $ G $ for every member $ P_{i} $ of some fixed  $ \MM_{d}(P) $.
  \end{theoremA}

   The conditions that $ G $ is $ p $-soluble in Theorem \ref{first}, $ N_{G}(P) $ is $ p $-nilpotent in Theorem \ref{second}, and $ (|G|, p-1)=1 $ in Theorem \ref{third} cannot be removed. Let us see the following example.

   \begin{example}
   	Let $ A=\langle a|a^{p}=1\rangle $ be a cyclic group of order $ p $, where $ p $ is a prime. Let $ B $ be a non-abelian simple group such that the Sylow $ p $-subgroups of $ B $ have order $ p $. Let $ b $ be an element of $ B $ of order $ p $. Set $ G=A\times B $. Let $ P $ be a Sylow $ p $-subgroup of $ G $ such that $ a, b\in P $. Write $ P_{1}=A $ and $ P_{2}=\langle ab\rangle $. It is easy to see that $ P_{i}\cap B $ satisfies the $ \Pi $-property in $ G $ for any $ i=1, 2 $. However, $ G $ is not $ p $-supersoluble or $ p $-nilpotent.
   \end{example}

\section{Preliminaries}

In this section, we collect some results which are needed in the proofs.

\begin{lemma}[{\cite[Proposition 2.1(1)]{li-2011}}]\label{over}
	Let $ H $ be a subgroup of a group $ G $ and $ N $ a normal subgroup of a group $ G $. If $ H $ satisfies the $ \Pi $-property in $ G $, then $ HN/N $ satisfies the $ \Pi $-property in $ G/N $.
\end{lemma}

\begin{lemma}\label{p-one}
	Let $ p $ be a prime dividing the order of a group $ G $ and $ P_{1} $ be a $ p $-subgroup of $ G $. Let $ L\unlhd G $ and $ N $ be a normal $ p' $-subgroup of $ G $. Then $ P_{1}N/N \cap LN/N = (P_{1} \cap L)N/N $.
\end{lemma}

\begin{lemma}[{\cite[Kapitel I, Hauptsatz 17.4]{Huppert-1967}}]\label{complemented}
	Suppose that $ N $ is an abelian normal subgroup of a group $ G $ and $ N \leq M \leq G $ such that $ (|N|, |G:M|)=1 $. If $ N $ is complemented in $ M $, then $ N $ is complemented in $ G $.
\end{lemma}

\begin{lemma}\label{p-supersoluble}
	Let $ H $ be a  $ p $-subgroup of a group $ G $ for some prime $ p\in \pi(G) $. If $ G $ is $ p $-supersoluble, then $ H $ satisfies the $ \Pi $-property in $ G $.
\end{lemma}
\begin{proof}[\bf{Proof}]
	By \cite[Lemma 2.3]{Su-2014}, the conclusion holds.
\end{proof}

\begin{lemma}[{\cite[Lemma 2.1(6)]{Li-Miao-2017}}]\label{cap}
	Let $ p $ be a prime dividing the order of a group $ G $, $ H $ be a $ p $-subgroup of $ G $ and $ N \unlhd G $. If $ H $ satisfies the $ \Pi $-property in $ G $, then $ H\cap N $ satisfies the $ \Pi $-property in $ G $.
\end{lemma}

\section{Proofs}

\begin{proof}[\bf{Proof of Theorem \ref{first}}]
	By Lemma \ref{p-supersoluble}, we only need to prove the sufficiency. Suppose that $ G $ is not $ p $-supersoluble and let $ G $ be a counterexample of minimal  order.  We divide the proof into the following steps.
	
	\vskip0.1in
	
	\textbf{Step 1.} $ O_{p'}(G)=1 $.

	Set $ \overline{G}=G/O_{p'}(G) $. Clearly, $ \overline{P} $ is a Sylow $ p $-subgroup of $ \overline{G} $, which is isomorphic to $ P $. Thus $ \MM_{d}(\overline{P})=\{\overline{P_{1}}, \overline{P_{2}}, ..., \overline{P_{d}} \} $.  By hypothesis,  $ P_{i}\cap N $ satisfies the $ \Pi  $-property in $ G $ for any $ P_{i} \in \MM_{d}(P) $.  By Lemmas \ref{over} and \ref{p-one},  $ \overline{P_{i}\cap N}=\overline{P_{i}}\cap \overline{N} $ satisfies the $ \Pi $-property in $ \overline{G} $ for any $ \overline{P_{i}}\in \MM_{d}(\overline{P}) $. Clearly, $ \overline{G}/\overline{N}\cong G/NO_{p'}(G) $ is $ p $-supersoluble and $ \overline{G} $ is $ p $-soluble. Thus $ \overline{G} $ satisfies the hypotheses of the theorem. If $ O_{p'}(G)>1 $, then $ G/O_{p'}(G) $ is $ p $-supersoluble by the minimal choice of $ G $. Hence $ G $ is $ p $-supersoluble, a contradiction. Therefore $ O_{p'}(G)=1 $.

	\vskip0.1in
	
	\textbf{Step 2.} $ \Phi(P)_{G}=1 $. In particular, $ \Phi(O_{p}(G))=1 $.

	Assume that $ \Phi(P)_{G}>1 $, and we work to obtain a contradiction. Let $ V $ be a minimal normal subgroup of $ G $ contained in $ \Phi(P)_{G} $.  Since every maximal subgroup of $ P $ contains $ \Phi(P) $ and $ P/V $ has the same smallest generator number as $ P $, we have $ \MM_{d}(P/V)=\{ P_{1}/V, P_{2}/V, ..., P_{d}/V\} $. By Lemma \ref{over} and Dedekind's lemma,  $ P_{i}/V\cap NV/V $ satisfies the $ \Pi $-property in $ G/V $ for any $P_{i}\in \MM_{d}(P)$. It is not  difficult to see that $ G/V $ satisfies the hypotheses of the theorem. The minimal choice of $ G $ yields that $ G/V $ is $ p $-supersoluble. Since $V\leq \Phi(P)$, it follows from \cite[Chapter A, Theorem 9.2(d)]{Doerk-Hawkes} that $ V\leq \Phi(G) $. Note that  the class of all $ p $-supersoluble groups is a saturated formation, and so $ G $ is $ p $-supersoluble, a contradiction. Thus $ \Phi(P)_{G}=1 $. By \cite[Chapter A, Theorem 9.2(e)]{Doerk-Hawkes}, we have $ \Phi(O_{p}(G))=1 $.
	
	\vskip0.1in
	
	\textbf{Step 3.} Every minimal normal subgroup of $ G $ contained in $ O_{p}(G) $ is of order $ p $.

	Since $ G $ is $ p $-soluble and $ O_{p'}(G)=1 $, we have $ O_{p}(G)>1 $. Let $ K $ be a minimal normal subgroup of $ G $ contained in $ O_{p}(G) $. By Step 2, there exists a member $ P_{j}\in \MM_{d}(P) $ such that $ K\nleq P_{j} $. By hypothesis, $ P_{j}\cap N $ satisfies the $ \Pi $-property in $ G $. For the $ G $-chief factor $ K/1 $, we see that $ |G:N_{G}(P_{j}\cap N\cap K)| $ is a $ p $-number, and so $ P_{j}\cap N\cap K\unlhd G $. If $ K\leq N $, then $ P_{j}\cap K=1 $ by the minimality of $ K $. Therefore $ K $ has order $ p $, as desired. Assume that $ K\nleq N $. Since $ O_{p'}(G)=1 $ and $ G/N $ is $ p $-supersoluble, we deduce that $ |K|=p $, as desired.

	\vskip0.1in
	
	\textbf{Step 4.}  The final contradiction.

	Since $ G $ is $ p $-soluble, by Step 1, Step 2 and \cite[Chapter 6, Theorem 3.2]{Gorenstein-1980}, we have $ C_{G}(O_{p}(G)) = O_{p}(G) $. Assume that $ O_{p}(G) \cap \Phi(G) >1 $. Let $ U $ be a minimal normal subgroup of $ G $ contained in $ O_{p}(G) \cap \Phi(G) $. Then $ U $ has order $ p $ by Step 3. In view of  Step 2,  $ U $ is complemented in $ P $. By Lemma \ref{complemented}, $ U $  is complemented in $ G $, which is contrary to $ U \leq \Phi(G) $. Therefore, $ O_{p}(G)\cap \Phi(G)=1 $. By Step 3 and \cite[Lemma 2.13]{skiba2007weakly}, we have that $ O_{p}(G) = U_{1} \times U_{2} \times \cdots \times U_{s} $, where $ U_{i} $ is a minimal normal subgroup of $ G $ of order $ p $ for any $ i=1,...,s $. Note that  $ G/C_{G}(U_{i}) $ is cyclic of order dividing $ p-1 $ for any $ i=1,...,s $. Hence $ G\big /\bigcap\limits_{i=1}^{s} C_{G}(U_{i})=G/C_{G}(O_{p}(G))=G/O_{p}(G) $ is abelian. Since every chief factor of $ G $ below $ O_{p}(G) $ is of order $ p $, we see that $ G $ is $ p $-supersoluble. This final contradiction completes the proof.
\end{proof}

\begin{proof}[\bf{Proof of Theorem \ref{second}}]
	By Lemma \ref{p-supersoluble}, we only need to prove the sufficiency. Suppose that $ G $ is not $ p $-nilpotent and let $ G $ be a counterexample of minimal  order.
	
	\vskip0.1in
	
	\textbf{Step 1.} $ O_{p'}(G)=1 $.

	Set $ \overline{G}=G/O_{p'}(G) $. 	Arguing as Step 1 of the proof of Theorem \ref{first}, we know that $ \overline{G} $ satisfies the hypotheses of the theorem. If $ O_{p'}(G)>1 $, then $ G/O_{p'}(G) $ is $ p $-nilpotent by the minimal choice of $ G $. Hence $ G $ is $ p $-nilpotent, a contradiction. Therefore $ O_{p'}(G)=1 $.

	\vskip0.1in
	
	\textbf{Step 2.} $ \Phi(P)_{G}=1 $. In particular, $ \Phi(O_{p}(G))=1 $.

	Assume that $ \Phi(P)_{G}>1 $. Let $ V $ be a minimal normal subgroup of $ G $ contained in $ \Phi(P)_{G} $. With a similar argument as Step 2 of the proof of Theorem \ref{first}, we can obtain that $ G/V $ is $ p $-nilpotent. Since $ V\leq \Phi(P) $, we have $ V\leq \Phi(G) $. Note that  the class of all $ p $-nilpotent groups is a saturated formation, and hence $ G $ is $ p $-nilpotent, a contradiction. Thus $ \Phi(P)_{G}=1 $, and so $ \Phi(O_{p}(G))=1 $.

	\vskip0.1in

	\textbf{Step 3.} Let $ K $ be a minimal normal subgroup of $ G $ such that $ K\nleq N $. Then $ |K|=p $.

	Note that $ KN/N $ is a minimal normal subgroup of $ G/N $. Since $ O_{p'}(G)=1 $ and $ G/N $ is $ p $-supersoluble, we deduce that $ |K|=p $, as wanted.
	
	\vskip0.1in
	
	\textbf{Step 4.} Every minimal normal subgroup of $ G $ is contained in $ O_{p}(G) $.

	Assume that $ K $ is a minimal normal subgroup of $ G $, which is not a $ p $-group. We work to obtain a contradiction. Since $ O_{p'}(G)=1 $, we have $ p\big ||K| $. If $ K\nleq N $, then $ |K|=p $ by Step 3, a contradiction.   Therefore $ K\leq N $. By hypothesis, $ P_{i} $ satisfies the $ \Pi $-property in $ G $ for any $ P_{i}\in \MM_{d}(P) $. Considering the $ G $-chief factor $ K/1 $, we have that $  |G:N_{G}(P_{i}\cap K)| $ is a $ p $-number, and thus $ P_{i}\cap K\unlhd G $. Then either $ P_{i}\cap K=1 $ or $ P_{i}\cap K=K $. If $ P_{i}\cap K=K $, then $ K $ is a $ p $-group, a contradiction. Therefore, $ P_{i}\cap K=1 $. Since $ O_{p'}(G)=1 $, it follows that $ |P\cap K|=p $. If $ K=G $, then $ P $ is cyclic of order $ p $. Since $ N_{G}(P) $ is $ p $-nilpotent, we deduce that $ P\leq Z(N_{G}(P)) $.  By Burnside's Theorem (see \cite[Chapter 7, Theorem 4.3]{Gorenstein-1980}),  it follows that $ G $ is $ p $-nilpotent, a contradiction. Therefore $ K<G $. Set $ L=PK $. Note that $ |PK|=\frac{|P||K|}{|P\cap K|}=|P_{i}||K| $ and $ P_{i}\cap K=1 $. Thus $ L=P_{i}K $. Observe that  $ L/K $ is a $ p $-group (also a $ p $-supersoluble group), $ P_{i}\cap K=1 $ satisfies the $ \Pi $-property in $ G $ and $ N_{L}(P) $ is $ p $-nilpotent. Hence $ L $ inherits the hypotheses of the theorem. If $ L<G $, then $ L $ is $ p $-nilpotent by the minimal choice of $ G $. Thus $ K $ is $ p $-nilpotent. Since $ O_{p'}(G)=1 $, we have $ K\leq O_{p}(G) $, a contradiction. Hence $ P_{i}K=L=G $. Write $ W=N_{G}(P\cap K) $. Clearly, $ P\leq W $, $ W=P_{i}K\cap W=P_{i}(K\cap W) $. Note that $ W/(K\cap W) $ is a $ p $-group, $ P_{i}\cap (K\cap W)=P_{i}\cap K=1 $ satisfies the $ \Pi $-property in $ G $ and $ N_{W}(P) $ is $ p $-nilpotent. Therefore $ W $ satisfies the hypotheses of the theorem. If $ W < G $, then $ W=N_{G}(P\cap K) $  is $ p $-nilpotent by the minimal choice of $ G $, and so $ N_{K}(P\cap K) $ is $ p $-nilpotent. Since $ |P\cap K|=p $, we have $ P\cap K\leq Z(N_{K}(P\cap K)) $. It follows that $ K $ is $ p $-nilpotent. By Step 1, we know $ K\leq O_{p}(G) $, a contradiction. If $W=G$, then $K=P\cap K$ by Step 1, a contradiction.  Hence Step 4 holds.

	\vskip0.1in
	
	\textbf{Step 5.} Every minimal normal subgroup of $ G $ is of order $ p $.

	Let $ K $ be a minimal normal subgroup of $ G $. By Step 4, we have $ K \leq O_{p}(G) $. If $ K\nleq N $, then $ |K|=p $ by Step 3, as desired. Hence we may assume that $ K\leq N $. By hypothesis, $P_{i}$ satisfies the $ \Pi $-property in $ G $ for any $ P_{i}\in \MM_{d}(P) $. For the $ G $-chief factor $ K/1 $, we know that $ |G:N_{G}(P_{i}\cap K)| $ is a $ p $-number, and so $ P_{i}\cap K\unlhd G $. The minimality of $K$ implies that either $ P_{i}\cap K=1 $ or $ P_{i}\cap K=K $, where $ i=1, ..., d $. If $ P_{i}\cap K=K $ for every  $ P_{i} \in \MM_{d}(P) $, then $ K\leq\mathop{\cap}\limits_{i=1}^{d} P_{i}=\Phi(P) $, which contradicts Step 2. So there exists a member $ P_{j}\in \MM_{d}(P) $ such that $ P_{j}\cap K=1 $. Since $ O_{p'}(G)=1 $, we have $ |P\cap K|=p $, and thus $ |K|=p $ by Step 4, as wanted.

	\vskip0.1in
	
	\textbf{Step 6.} The final contradiction.

	Let $ R $ be a subgroup of $ G $ of minimal order such that $ O_{p}(G)R=G $.  Assume that $ O_{p}(G) \cap R >1 $. By Step 2, we see that $ O_{p}(G)\cap R \unlhd G $. Let $ T $ be a minimal normal subgroup of $ G $ contained in $ O_{p}(G)\cap R $. Then $T$ has order $p$ by Step 5. In view of Step 2, we see that $ T $ is complemented in $ P $. By Lemma \ref{complemented}, $ T $ has a complement $ A $ in $ G $. Hence $ R=R\cap TA=T(R\cap A) $, and so $ G=O_{p}(G)R=O_{p}(G)(R\cap A) $. The minimal choice of $ R $ yields that $ R\cap A=R $, i.e., $ R\leq A $. As a consequence, $ T\leq O_{p}(G)\cap R\leq A $. This is a contradiction. Therefore $ O_{p}(G)\cap R=1 $.
	
	Assume that $ O_{p}(G) \cap \Phi(G) > 1 $. Let $ U $ be a minimal normal subgroup of $ G $ contained in $ O_{p}(G) \cap \Phi(G) $. Then $ U $ has order $ p $ by Step 5. In view of Step 2,  $ U $ is complemented in $ P $. By Lemma \ref{complemented}, $ U $  is complemented in $ G $, which contradicts the fact that $ U \leq \Phi(G) $. Therefore, $ O_{p}(G)\cap \Phi(G)=1 $. By Step 5 and \cite[Lemma 2.13]{skiba2007weakly},  $ O_{p}(G) = U_{1} \times U_{2} \times \cdots \times U_{s} $, where $ U_{i}\unlhd G $ and $ |U_{i}| = p $ for any $ i=1,...,s $. Hence $ O_{p}(G)\leq Z(P) $. If $ P\cap R=1 $, then $  P =O_{p}(G)(P\cap R)=O_{p}(G) $, and thus $ G=N_{G}(P) $ is $ p $-nilpotent, a contradiction. Therefore, $ P\cap R>1 $. Let $ B $ be a minimal normal subgroup of $ G $ contained in $ (P \cap R)^{G} $. By Step 4, we have $ B\leq O_{p}(G) $. Since $O_{p}(G)\leq Z(P)$, we have  $ (P \cap R)^{G}= (P \cap R)^{O_{p}(G)R} =(P\cap R)^{R}\leq R $. Thus $ B\leq R $, which contradicts the fact that $ O_{p}(G)\cap R=1 $. This  final contradiction completes the proof.
\end{proof}

\begin{proof}[\bf{Proof of Theorem \ref{third}}]
	By Lemma \ref{p-supersoluble}, we only need to prove the sufficiency. Suppose that $ G $ is not $ p $-nilpotent and let $ G $ be a counterexample of minimal  order.
	
	\vskip0.1in
	
	\textbf{Step 1.} $ O_{p'}(G)=1 $.
	
    With a similar argument as Step 1 of the proof of Theorem \ref{second}, we can get that $ O_{p'}(G)=1 $.

	\vskip0.1in
	
	\textbf{Step 2.} $ \Phi(P)_{G}=1 $. In particular, $ \Phi(O_{p}(G))=1 $.

	Arguing as Step 2 of the proof of Theorem \ref{second}, we see that  $ \Phi(P)_{G}=1 $ and $ \Phi(O_{p}(G))=1 $.
	
	

	
	\vskip0.1in
	
	\textbf{Step 3.} Every minimal normal subgroup of $ G $ has order $ p $.

	
	Let $ K $ be a minimal normal subgroup of $ G $. Since $ O_{p'}(G)=1 $, we have $ p\big ||K| $. If $ K\nleq N $, then $ KN/N $ is a minimal normal subgroup of $ G/N $. Since $ O_{p'}(G)=1 $ and $ G/N $ is $ p $-supersoluble, we deduce that $ |K|=p $, as desired. Hence we may assume that $ K\leq N $. By hypothesis,  $ P_{i}\cap N $ satisfies the $ \Pi $-property in $ G $ for any $ P_{i}\in \MM_{d}(P) $. Considering the $ G $-chief factor $ K/1 $, we see that  $ |G:N_{G}(P_{i}\cap K)| $ is a $ p $-number, and thus $ P_{i}\cap K\unlhd G $. The minimality of $ K $ implies that either $ P_{i}\cap K=1 $ or $ P_{i}\cap K=K $, where $ i=1, ..., d $. If $ P_{i}\cap K=K $ for any $ P_{i} \in \MM_{d}(P) $, then $ K\leq\mathop{\cap}\limits_{i=1}^{d} P_{i}=\Phi(P) $, which contradicts Step 2. So there exists a member $ P_{j}\in \MM_{d}(P) $ such that $ P_{j}\cap K=1 $. In view of $ O_{p'}(G)=1 $, we have $ |P\cap K|=p $. Since $ (|G|, p-1)=1 $, it follows from  \cite[Chapter 1, Lemma 3.39]{Guo2015} that $ K $ is $ p $-nilpotent. Again, $ O_{p'}(G)=1 $, we conclude that $ |K|=p $, as desired.

	\vskip0.1in
	
	\textbf{\noindent Step 4.} The final contradiction.

	Let $ T_{1}, T_{2}, ..., T_{s} $ be all minimal normal subgroups of $ G $. By Step 3, $ T_{i} $ is of order $ p $ for any $ i=1, ..., s $. By Step 2, $ T_{i} $ is complemented in $ P $, and so $ T_{i} $ has a complement $ X_{i} $ in $ G $ for any $ i=1, ..., s $ by Lemma \ref{complemented}. Note that $ G/C_{G}(T_{i}) \lesssim \mathrm{Aut}(T_{i}) $. Since $ (|G|, p-1)=1 $, it follows that $ G=C_{G}(T_{i}) $, i.e., $ T_{i}\leq  Z(G) $ for any $ i=1, ..., s $. Thus $ X_{i}\unlhd G $ for any $ i = 1, 2, ..., s $. Let $ Y $ be a subgroup of $ G $ of minimal order such that $ (T_{1}T_{2}\cdots T_{s})Y=G $.   Assume that  $ O_{p}(G)\cap Y >1 $. By Step 2, we see that $ O_{p}(G)\cap Y  \unlhd G $.  Let $ U $ be a minimal normal subgroup  of $ G $ contained in $ O_{p}(G)\cap Y $. Then $ U = T_{j} $ for some $ j\in \{ 1, 2, ..., s\} $. Hence $ G=T_{j}X_{j}=UX_{j} $ and $ Y=U(Y\cap X_{j}) $. As a consequence, $ G=(T_{1}T_{2}\cdots T_{s})Y=(T_{1}T_{2}\cdots T_{s})(Y\cap X_{j}) $.  The minimal choice of $ Y $ implies that $ Y\cap X_{j}=Y $, and thus $ T_{j}=U\leq  Y\leq X_{j} $, a contradiction.  Therefore $ O_{p}(G)\cap Y =1 $. Thus $ O_{p}(G)=T_{1}T_{2}\cdots T_{s}\leq Z(G) $. It follows that  $ 1<Y\unlhd G $. By Step 3, we have $ O_{p}(G)\cap Y>1 $, a contradiction. The proof is now complete.
\end{proof}

\section{Final remarks and applications}\label{S4}

In this section, we will show that the concept of the $ \Pi $-property can be viewed as a generalization of many known embedding properties.

Recall that two subgroups $ H $ and $ K $ of a group $ G $ are said to be \emph{permutable}  if $ HK = K H $.  From Kegel \cite{kegel1962sylow},  a subgroup $ H $ of a group $ G $ is said to be \emph{$ S $-permutable}  (or \emph{$\pi$-quasinormal}, \emph{$ S $-quasinormal}) in $ G $ if $ H $ permutes with all Sylow subgroups of $ G $.  According to \cite{guo2007x}, a subgroup $ H $ of a group $ G $  is said to be \emph{$ X $-permutable}  with a subgroup $ T $ of $ G $  if there is an element $ x \in X $ such that $ HT^{x} = T^{x}H $, where  $ X $ is a   non-empty subset of $ G $.   Following  \cite{Alsheik}, the $ \UU $-hypercenter $ Z_{\UU}(G) $ of a group $ G $ is the product of all normal subgroups $ H $ of $ G $, such that all $ G $-chief factors below $ H $ have prime order, where  $ \UU $ denotes the class of all
supersoluble groups. A subgroup $ H $ of a group $ G $ is called a \emph{CAP-subgroup} of $ G $ if $ H $ either covers or avoids every chief factor $ L/K $ of $ G $, that is, $ HL =HK $ or $ H \cap L =H \cap K $ (see \cite[Chapter A, Definition 10.8]{Doerk-Hawkes}).

\begin{proposition}\label{p1}
	Let $ H $ be a subgroup of a group $ G $. Then $ H $ satisfies $ \Pi $-property in $ G $, if one of the following holds:
	
	\vskip0.08in
	\noindent{\rm(1)} $ H $ is normal in $ G $;
	
	\noindent{\rm(2)} $ H $ is permutable in $ G $;
	
	\noindent{\rm(3)} $ H $ is $ S $-permutable in $ G $;
	
	\noindent{\rm(4)} $ H $ is $ X $-permutable with all Sylow subgroups of $ G $, where $ X $ is a soluble normal subgroup of $ G $;
	
	\noindent{\rm(5)} $ H $ is a CAP-subgroup of $ G $;
	
	\noindent{\rm(6)} $ H/H_{G}  \leq Z_{\UU}(G /H_{G}) $.
\end{proposition}

\begin{proof}[\bf{Proof}]
	Statements (1)-(6) were proved in \cite[Propositions 2.2-2.3]{li-2011}.
\end{proof}

A subgroup $ H $ of a group $ G $ is said to be \emph{$ S $-semipermutable}  \cite{Chen-1987} in $ G $  if $ HG_{p} = G_{p}H $ for any Sylow $ p $-subgroup $ G_{p} $ of $ G $ with $ (p,|H|) = 1 $. A subgroup $ H $ of a group $ G $ is  said to be \emph{$ SS $-quasinormal}  \cite{Li-2008} in  $ G $ if there is a subgroup  $ B $ of $ G $  such that $ G=HB $ and  $ H $ permutes with every Sylow subgroup of $ B $.

\begin{proposition}\label{p2}
	Let $ H $ be a $ p $-subgroup of a group $ G $  for some prime $ p \in \pi( G ) $. Then $ H $ satisfies the $ \Pi $-property in $ G $, if one of the following holds:
	
	\vskip0.08in
	\noindent{\rm (1)} $ H $ is $ S $-semipermutable in $ G $;
	
	\noindent{\rm (2)} $ H $ is $ SS $-quasinormal in  $ G $.
	
\end{proposition}

\begin{proof}[\bf{Proof}]
	(1) By \cite[Proposition 2.4]{li-2011}, the conclusion follows.
	
	(2) Applying  \cite[Lemma 2.5]{Li-2008}, we see that $ H $ is $ S $-semipermutable in $ G $. Therefore, $ H $ satisfies the $ \Pi $-property in $ G $ by statement (1).
\end{proof}

A subgroup $ H $ of a group $ G $ is said to be \emph{$ \Pi $-normal}  \cite{li-2011} in $ G $ if $ G $ has a subnormal subgroup $ T $ such that $ G = HT $ and $ H \cap T\leq  I \leq H $, where $ I $ satisfies the $ \Pi $-property in $ G $.
A   subgroup $ H $ of a group $ G $ is called \emph{$ c $-normal}   \cite{wang1996c} in $ G $ if $ G $ has a normal subgroup $ T $ such that  $ HT = G $ and $ T \cap H \leq H_{G} $,  where $ H_{G} $ is the core of $ H $ in $ G $.  A subgroup $ H $ of a group $ G $ is called \emph{$ \UU_{c} $-normal}  \cite{Alsheik} in $ G $  if there is a subnormal subgroup $ T $ of $ G $ such that $ G=HT $ and  $ (H\cap T)H_{G}/H_{G}  \leq Z_{\UU}(G /H_{G}) $.

A subgroup $ H $ of a group $ G $ is said to be \emph{weakly $ S $-permutable} \cite{skiba2007weakly} in $ G $ if there is a subnormal subgroup $ T $ of $ G $ such that $ G = HT $ and $ H \cap T \leq H_{sG}  $, where $ H_{sG} $ is the subgroup of $ H $ generated by all those subgroup of $ H $ which are $ S $-permutable in $ G $. A subgroup $ H $ of a group $ G $ is said to be \emph{weakly $ SS $-permutable}  \cite{He-2010} in $ G $ if there exist a subnormal subgroup $ T $ of $ G $ and an $ SS $-quasinormal subgroup $ H_{ss} $ of $ G $ contained in $ H $ such that $ G= HT $ and $ H\cap T\leq H_{ss} $. A subgroup $ H $ of a group $ G $ is called a \emph{$ c^{\sharp} $-normal}  subgroup \cite{wang2013c} of $ G $ if $ G $ has a normal subgroup $ T $ such that $ G = HT $ and $ H \cap T $ is a CAP-subgroup of $ G $. A subgroup $ H $  of a group $ G $ is \emph{weakly $ S $-semipermutable}  \cite{li2012weakly} in $ G $ if there are a subnormal subgroup $ T $ of $ G $ and an $ S $-semipermutable subgroup $ H_{ssG} $ of $ G $ contained in $ H $ such
that $ G = HT $ and $ H \cap T \leq H_{ssG} $.

\begin{proposition}\label{p3}
	Let $ H $ be a $ p $-subgroup of a group $ G $  for some prime $ p \in \pi( G ) $. Then $ H\cap O^{p}(G) $ satisfies the $ \Pi $-property in $ G $, if one of the following holds:
	
	\vskip0.08in
	\noindent{\rm(1)}  $ H $ is $ \Pi $-normal in $ G $;
	
	\noindent{\rm(2)} $ H $ is $ c $-normal in  $ G $;
	
	\noindent{\rm(3)} $ H $ is $ \UU_{c} $-normal in $ G $;
	
	\noindent{\rm(4)} $ H $ is weakly $ S $-permutable in $ G $;
	
	\noindent{\rm(5)} $ H $ is a $ c^{\sharp} $-normal subgroup of $ G $;
	
	\noindent{\rm(6)} $ H $ is weakly $ SS $-permutable in $ G $;
	
	\noindent{\rm(7)} $ H $ is weakly $ S $-semipermutable in $ G $.
\end{proposition}

\begin{proof}[\bf{Proof}]
	(1) If $ H $ is $ \Pi $-normal in $ G $, then $ G $ has a subnormal subgroup $ T $ such that $ G =HT $ and $ H \cap T\leq  I \leq H $, where $ I $ satisfies the $ \Pi $-property in $ G $. Since $ H $ is a $ p $-subgroup of $ G $, it follows from \cite[Lemma 1.1.11]{MR2762634} that  $ O^{p}(T) = O^{p}(G) $. Then $ H\cap O^{p}(G)=I\cap O^{p}(G) $ satisfies the $  \Pi $-property in $ G $ by Lemma \ref{cap}.
	
	Statements (2)-(7) directly follow from statement (1). By Proposition \ref{p1},  a $ c $-normal  (respectively, $ \UU_{c} $-normal,  weakly $ S $-permutable, $ c^{\sharp} $-normal)  subgroup of $ G $ is a $ \Pi $-normal subgroup of  $ G $.
	If $ H $ is a  weakly $ SS $-permutable (respectively, weakly $ S $-semipermutable) $ p $-subgroup of $ G $, then $ H $ is $ \Pi $-normal in $ G $ by Proposition \ref{p2}.
\end{proof}

Applying Propositions \ref{p1}, \ref{p2} and \ref{p3},  we can obtain the following  corollaries.

\begin{corollary}[{\cite[Theorem 1.1]{Li-2008}}]
	Let $ p $ be the smallest prime dividing the order of a group $ G $ and $ P $ a Sylow
	$ p $-subgroup of $ G $. If every member of some fixed $ \MM_{d}(P) $ is $ SS $-quasinormal in $ G $, then $ G $ is $ p $-nilpotent.
\end{corollary}

\begin{corollary}[{\cite[Theorem 1.2]{Li-2008}}]
	Let $ p $ be a prime dividing the order of a group $ G $ and $ P $ a Sylow $ p $-subgroup of $ G $. If $ N_{G}(P) $ is $ p $-nilpotent and every member of some fixed $ \MM_{d}(P) $ is $ SS $-quasinormal in $ G $, then $ G $ is $ p $-nilpotent.
\end{corollary}

\begin{corollary}[{\cite[Theorem 1.3]{Li-2008}}]
	Let $ G $ be a $ p $-solvable group for a prime $ p $ and $ P $ a Sylow $ p $-subgroup of $ G $.
	Suppose that every member of some fixed  $ \MM_{d}(P) $ is $ SS $-quasinormal in $ G $. Then $ G $ is $ p $-supersoluble.
\end{corollary}

\begin{corollary}[{\cite[Theorem 3.1]{LiYangming-2010}}]
	Let $ G $ be a group and $ P $ be  a Sylow $ p $-subgroup of $ G $, where $ p $ is the smallest prime dividing $ |G| $. If all maximal subgroups of $ P $ are $ S $-semipermutable in $ G $, then $ G $ is $ p $-nilpotent.
\end{corollary}

\begin{corollary}[{\cite[Theorem 3.1]{Liu-2008}}]
	Let $ G $ be a $ p $-soluble group and let $ P $ be a Sylow $ p $-subgroup
	of $ G $, where $ p $ is a fixed prime. Then  $ G $ is $ p $-supersoluble if  and only if every member of some fixed  $ \MM_{d}(P) $ is a CAP-subgroup of $ G $.
\end{corollary}

\begin{corollary}[{\cite[Theorem 3.3]{Liu-2008}}]
	Let $ p $ be the smallest prime dividing the order of a group $ G $ and let $ P $ be a Sylow $ p $-subgroup of $ G $. Then $ G $ is $ p $-nilpotent if and only if  every member of some fixed  $ \MM_{d}(P) $ is a CAP-subgroup of $ G $.
\end{corollary}

\begin{corollary}[{\cite[Theorem 3.4]{Liu-2008}}]
	Suppose that $ P $ is a Sylow $ p $-subgroup of a group $ G $ and $ N_{G}(P) $ is $ p $-nilpotent for some prime $ p\in \pi(G) $. Then $ G $ is $ p $-nilpotent
	if and only if  every member of some fixed  $ \MM_{d}(P) $ is a CAP-subgroup of $ G $.
\end{corollary}

\begin{corollary}[{\cite[Theorem 3.8]{Lu-2009}}]
	Let $ p $ be a prime dividing the order of a $ p $-soluble group $ G $ and let $ P $ be a Sylow $ p $-subgroup of $ G $. If every member of some fixed  $ \MM_{d}(P) $ is $ S $-semipermutable in $ G $, then $ G $ is $ p $-supersoluble.
\end{corollary}

\begin{corollary}[{\cite[Theorem 3.9]{Lu-2009}}]
	Let $ p $ be an odd prime dividing the order of $ G $ and let $ P $ be a Sylow $ p $-subgroup of a group $ G $. If $ N_{G}(P) $ is $ p $-nilpotent and every member of some fixed  $ \MM_{d}(P) $ is $ S $-semipermutable in $ G $, then $ G $ is $ p $-nilpotent.
\end{corollary}

\begin{corollary}[{\cite[Theorem 3.1]{Zhong-Lin}}]
	Let $ G $ be a group, and let $ P $ be a Sylow $ p $-subgroup of $ G $, where
	$ p $ is a prime divisor of $ |G| $ with $ (|G|, p-1) = 1 $. Then $ G $ is $ p $-nilpotent if and only if every member in some fixed $ \MM_{d}(P) $ is $ c^{\sharp} $-normal in $ G $.
\end{corollary}

\begin{corollary}[{\cite[Theorem 3.2]{Zhong-Lin}}]
	Let $ G $ be a $ p $-soluble group, and let $ P $ be a Sylow $ p $-subgroup
	of $ G $, where $ p $ is a prime divisor of $ |G| $. Then $ G $ is $ p $-supersolvable if and only if every member in some fixed $ \MM_{d}(P) $ is $ c^{\sharp} $-normal in $ G $.
\end{corollary}

\section*{Acknowledgments}

This work is supported by the National Natural Science Foundation of China (Grant No.12071376).
	
\small

\end{sloppypar}	
\end{document}